\documentclass[12pt]{amsart}
\usepackage{amsmath,amsfonts,amssymb,amsthm}

\theoremstyle{plain}
\newtheorem{theorem}{Theorem}

\newtheorem{proposition}[theorem]{Proposition}

\theoremstyle{definition}

\newcommand{\Z}{{\mathbb Z}}
\newcommand{\R}{{\mathbb R}}

\begin{document}

\title{An asymptotically sharp form of Ball's inequality by probability methods}
\author{
 Susanna Spektor}

\address{ Susanna Spektor,
\noindent Address: University of Alberta, Canada}
\email{sanaspek@gmail.com}


\thanks{}
\date{}
\maketitle

\thispagestyle{empty}

\setcounter{page}{1}


\section{Introduction}

To prove by probabilistic methods that every $(n-1)$-dimensional section of the unit cube in $\R^n$ has volume at most $\sqrt 2$, Ball \cite{Ball:1986} made essential use of the inequality
\begin{equation}\label{ball's ineq}
\frac{1}{\pi}\int_{-\infty}^{\infty} \left(\frac{\sin^2 t}{t^2}\right)^pdt\leq \frac{\sqrt 2}{\sqrt p}, \quad p\geq 1,
\end{equation}
 in which equality holds if and only if $p=1$.

 As we will see, the right side of (\ref{ball's ineq}) has the correct rate of decay though the limit of the ratio of the right to left side is $\displaystyle{\sqrt{\frac{3}{\pi}}}$ rather then $\sqrt 2$. Applying Ball's methods we put all of this into the followingimproved form of (\ref{ball's ineq}).

\begin{theorem}
Let
\begin{align*}
C(p):=
\left\{ \begin{array}{rcl}
\sqrt{\frac{3}{\pi}},&\quad 1\leq p \leq p_0\\
1+\frac{1}{\sqrt{3 \pi}}\frac{\left(\sqrt 5/6\right)^{2p-1}}{\sqrt p-1/2 \sqrt p},&\quad p> p_0,
\end{array}\right.
\end{align*}
where
\begin{align*}
\frac{\left(\sqrt 5/6\right)^{2p_0-1}}{\sqrt p_0-1/2 \sqrt p_0}=\left(1-\sqrt{3/{\pi}}\right) \pi
\end{align*}
so that $p_0=1.8414$ to $4$ D.

Than,
\begin{equation}\label{ineq theorem}
\frac{1}{\pi}\int_{-\infty}^{\infty} \left(\frac{\sin^2 t}{t^2}\right)^pdt\leq C(p) \frac{\sqrt{3/\pi}}{\sqrt p}\leq \frac{1}{\sqrt p}, \quad p\geq 1,
\end{equation}
the first two terms being equal if and only if $p=1$.

Further,
\begin{align*}
\lim_{p\longrightarrow \infty}\frac{1}{\pi}\int_{-\infty}^{\infty} \left(\frac{\sin^2 t}{t^2}\right)^pdt / \frac{\sqrt{3/\pi}}{\sqrt p}\leq \lim_{p \longrightarrow \infty} C(p)=1.
\end{align*}
\end{theorem}

\section{Symmetric B-splines and the integral $\displaystyle{\int_{-\infty}^{\infty} \left(\frac{\sin^2 t}{t^2}\right)^pdt}$}

The symmetric $B$-splines, $\beta^n$, are defined inductively by
\begin{align*}
\beta^0(x):=\chi_{[-\frac 12, \frac 12]}(x) \quad \mbox{and} \quad \beta^n(x):=\int_0^1 \beta^{n-1}(x-y)
dy,
\end{align*}
$n=1,2,...$

Using known properties of these $B$-splines we obtain an asymptotic formula for our integral as $p\longrightarrow \infty$, namely

\begin{proposition}
\begin{equation}\label{lim theorem}
\frac{1}{\pi}\int_{-\infty}^{\infty} \left(\frac{\sin^2 t}{t^2}\right)^pdt \sim \frac{\sqrt{3/\pi}}{\sqrt p},  \quad \mbox{ as } \quad p\longrightarrow \infty.
\end{equation}
\end{proposition}
\begin{proof}
Suppose to begin with that $p \in \Z_+$, say $p=n$. Now,
\begin{align*}
\widehat{\beta^n}(t):=\int_{-\infty}^{\infty}\beta^n(s)e^{-2 \pi i t s}ds=\left(\frac{\sin \pi t}{\pi t}\right)^n,
\end{align*}
so Plancherel's theorem yields
\begin{align*}
\frac{1}{\pi}\int_{-\infty}^{\infty} \left(\frac{\sin^2 t}{t^2}\right)^ndt= \int_{- \infty}^{\infty}\left(\frac{\sin t}{t}\right)^{2n}dt=\int_{-\infty}^{\infty}|\beta^n(s)|^2 ds.
\end{align*}
Further, by \cite{Chui:book},
\begin{align*}
\int_{-\infty}^{\infty}{\beta^n}(s)^2 ds=\int_{-\infty}^{\infty}\beta_n(s)\beta_n(1)^n ds=\beta^{2n}(0).
\end{align*}
Again, according to Theorem 1 in \cite{Unser:1992},
\begin{align*}
{\beta^{2n}}\left(\sqrt{\frac{2n+1}{12}}x\right) \sim \sqrt{\frac{6}{\pi(2n+1)}} \exp(-x^2/2),
\end{align*}
so in particular,
\begin{align*}
\beta^{2n}(0) \sim \sqrt{\frac{6}{\pi(2n+1)}} \sim \frac{\sqrt{3/ \pi}}{\sqrt{\pi}}, \quad \mbox{as} \quad p\longrightarrow \infty.
\end{align*}

Finally, $\displaystyle{\int_{-\infty}^{\infty}\left(\frac{\sin^2 t}{t^2}\right)^p dt}$ is a decreasing function of $p$, so one has
\begin{equation}\label{eq 4}
\frac{1}{\pi}\int_{-\infty}^{\infty}\left(\frac{\sin^2 t}{t^2}\right)^{[p]+1} dt\leq \frac{1}{\pi}\int_{-\infty}^{\infty}\left(\frac{\sin^2 t}{t^2}\right)^{p} dt\leq \frac{1}{\pi}\int_{-\infty}^{\infty}\left(\frac{\sin^2 t}{t^2}\right)^{[p]}dt
\end{equation}
and hence (\ref{lim theorem}), since the extreme term in (\ref{eq 4}) are both asymptotically equal to $\displaystyle{\frac{\sqrt{3/ \pi}}{\sqrt{p}}}$.
\end{proof}
\section{Proof of Theorem 1}
Ball shows that
\begin{align*}
\frac{1}{\pi}\int_{-6/\sqrt{5}}^{6/\sqrt{5}}\left(\frac{\sin^2 t}{t^2}\right)^p dt \leq \frac{\sqrt{3/ \pi}}{\sqrt{p}}.
\end{align*}
Further,
\begin{align*}
\frac{1}{\pi}\int_{|t|\geq 6/\sqrt{5}}^{\infty}\left(\frac{\sin^2 t}{t^2}\right)^p dt \leq \frac{2}{\pi}\int_{\geq 6/\sqrt{5}}^{\infty} t^{-2p} dt= \frac{1}{\pi}\frac{(\sqrt{5}/6)^{2p-1}}{p-\frac12}.
\end{align*}
Altogether, then,
\begin{align*}
\frac{1}{\pi}\int_{-\infty}^{\infty}\left(\frac{\sin^2 t}{t^2}\right)^p dt \leq \left(1+ \frac{1}{\sqrt{3 \pi}} \frac{(\sqrt 5/ 6)^{2p-1}}{\sqrt p- 1/2 \sqrt p}\right) \frac{\sqrt{3/ \pi}}{\sqrt {\pi}}.
\end{align*}
Finally,
\begin{align*}
1+ \frac{1}{\sqrt{3 \pi}} \frac{(\sqrt 5/ 6)^{2p-1}}{\sqrt p- 1/2 \sqrt p}\leq \sqrt{\pi/3}, \quad \mbox{ for } \quad p\geq p_0. \hspace{2cm} \Box
\end{align*}


\end{document}